\documentclass[reqno, 12pt]{amsart}
\usepackage{amssymb}
\usepackage{amsfonts}
\usepackage{srcltx}
\usepackage{amsmath}

\newtheorem{theorem}{Theorem}[section]
\newtheorem{lemma}[theorem]{Lemma}

\newtheorem{corollary}[theorem]{Corollary}
\theoremstyle{definition}
\newtheorem{definition}[theorem]{Definition}
\newtheorem{example}[theorem]{Example}

\newtheorem*{acknowledgement}{Acknowledgement}

\theoremstyle{remark}
\newtheorem{remark}[theorem]{Remark}

\DeclareMathOperator{\diam}{diam}

\DeclareMathOperator{\Fix}{Fix}

\begin{document}
\title[Uniformly Lipschitzian group actions]{Uniformly Lipschitzian group
actions \\
on hyperconvex spaces}
\author[A. Wi\'{s}nicki]{Andrzej Wi\'{s}nicki}
\author[J. Wo\'{s}ko]{Jacek Wo\'{s}ko}

\begin{abstract}
Suppose that $\{T_{a}:a\in G\}$ is a group of uniformly $L$-Lipschitzian
mappings with bounded orbits $\left\{ T_{a}x:a\in G\right\} $ acting on a
hyperconvex metric space $M$. We show that if $L<\sqrt{2}$, then the set of
common fixed points $\Fix G$ is a nonempty H\"{o}lder continuous retract of $%
M.$ As a consequence, it follows that all surjective isometries acting on a
bounded hyperconvex space have a common fixed point. A fixed point theorem
for $L$-Lipschitzian involutions and some generalizations to the case of $%
\lambda $-hyperconvex spaces are also given.
\end{abstract}

\address{Andrzej Wi\'{s}nicki, Department of Mathematics, Rzesz\'{o}w
University of Technology, Al. Powsta\'{n}c\'{o}w Warszawy 12, 35-959 Rzesz%
\'{o}w, Poland}
\email{awisnicki@prz.edu.pl}
\address{Jacek Wo\'{s}ko, Institute of Mathematics, Maria Curie-Sk{\l }%
odowska University, 20-031 Lublin, Poland}
\email{jwosko@hektor.umcs.lublin.pl}
\date{}
\maketitle

\section{Introduction}

A metric space $(M,d)$ is called hyperconvex if
\begin{equation*}
\bigcap_{\alpha \in \Gamma }B(x_{\alpha },r_{\alpha })\neq \emptyset
\end{equation*}%
for any collection of closed balls $\{B(x_{\alpha },r_{\alpha })\}_{\alpha
\in \Gamma }$ such that $d(x_{\alpha },x_{\beta })\leq r_{\alpha }+r_{\beta
},\ \alpha ,\beta \in \Gamma .$

The space $C(S)$ of continuous real functions on a Stonian space with the
\textquotedblleft supremum\textquotedblright\ norm is hyperconvex and every
hyperconvex real Banach space is $C(S)$ for some Stonian space. The standard
examples of hyperconvex spaces are $\ell _{\infty },L_{\infty }$ and their
unit balls.

The terminology is due to N. Aronszajn and P. Panitchpakdi \cite{ArPa} who
proved that a hyperconvex space is a nonexpansive retract of any metric
space in which it is isometrically embedded. J. Isbell \cite{Is} showed that
every metric space is isometric to a subspace of a unique \textquotedblleft
minimal\textquotedblright\ hyperconvex space called the injective envelope.
This notion was later rediscovered by A. Dress in \cite{Dr} as the tight
span in the context of optimal networks and phylogenetic analysis. For a
deeper discussion of hyperconvex spaces we refer the reader to \cite{EsFe,
EsKh, EsLo}.

In 1979, R. Sine \cite{Si} and P. Soardi \cite{So} showed independently that
nonexpansive mappings (i.e., mappings which satisfy $d(Tx,Ty)\leq
d(x,y),~x,y\in M$) defined on a bounded hyperconvex space has fixed points.
Since then, a number of fixed-point results in hyperconvex spaces were
obtained of both topological and metric character. In particular, J. Baillon
\cite{Ba} showed that any intersection of hyperconvex spaces with a certain
finite intersection property is a nonempty hyperconvex space. As a
consequence, he proved that the set of common fixed points of a commuting
family of nonexpansive mappings acting on a bounded hyperconvex space is a
nonexpansive retract of $M$.

In this paper we focus on properties of fixed-point sets of uniformly
Lipschitzian group actions on hyperconvex and\ $\lambda $-hyperconvex
spaces. Uniformly Lipschitzian mappings, introduced in \cite{GoKi1}, are
natural generalization of nonexpansive mappings. We say that a mapping $%
T:M\rightarrow M$ is uniformly $L$-Lipschitzian if $d(T^{n}x,T^{n}y)\leq
Ld(x,y)$ for each $x,y\in M$ and $n\in \mathbb{N}$. For example,
Lipschitzian periodic mappings are uniformly Lipschitzian. E. A. Lifschitz
\cite{Li} proved that if $C$ is a bounded, closed and convex subset of a
Hilbert space and $L<\sqrt{2}$, then every uniformly $L$-Lipschitzian
mapping $T:C\rightarrow C$ has a fixed point. From the geometric point of
view, Hilbert and hyperconvex spaces are two extremes and yet there are some
similarities between them. It was proved in \cite{LiXu} that every uniformly
$L$-Lipschitzian mapping with $L<\sqrt{2}$ in a bounded hyperconvex space
with the so-called property $(P)$ has a fixed point. This result was later
generalized in \cite{DKS} but a question whether a counterpart of Lifshitz's
theorem holds in every bounded hyperconvex space remains open.

We partially answer this question in Theorem \ref{1} in the case of
uniformly Lipschitzian group actions. Note that we assume the boundedness of
orbits only, instead of the boundedness of a target space. Theorem \ref{2}
generalizes the result to the case of $\lambda $-hyperconvex spaces. As a
consequence, we obtain the following rather surprising result, proved
recently by U. Lang \cite{La}: there exists a common fixed point for all
surjective isometries acting on a bounded hyperconvex space. We also give a
fixed point theorem for $L$-Lipschitzian involutions (Theorem \ref{3}). Note
that M. S. Brodski\u{\i} and D. P. Mil'man \cite{BrMi} proved a similar
statement for all surjective isometries acting on a weakly compact, convex
subset of a Banach space with normal structure.

In Section 3 we give a qualitative complement to Theorem \ref{1}: if $%
\{T_{a}:a\in G\}$ is a group of uniformly $L$-Lipschitzian mappings on a
hyperconvex space $M$ with $L<\sqrt{2}$ and the orbits are bounded, then $%
\Fix G$ is a H\"{o}lder continuous retract of $M.$

We begin with some basic definitions and notation. Let $\left( M,d\right) $
be a metric space. A semigroup $\left( S,\cdot \right) $ is said to act on $%
M $ (from the left) if there is a map%
\begin{equation*}
\varphi :S\times M\rightarrow M
\end{equation*}%
such that
\begin{equation*}
\varphi \left( a,\varphi \left( b,x\right) \right) =\varphi \left( a\cdot
b,x\right)
\end{equation*}%
for all $a,b\in S$ and $x\in M$. If $S$ has the identity element $e$, we
further assume that $\varphi \left( e,x\right) =x,~x\in M.$

\begin{enumerate}
\item If $T:M\rightarrow M$ is a mapping, then $\left( \mathbb{N},+\right) $
acts on $M$ by
\begin{equation*}
\varphi \left( 0,x\right) =x,\ \varphi \left( n,x\right) =T^{n}x.
\end{equation*}

\item If $T:M\rightarrow M$ is a bijection, then we get in a similar way an
action of a group $\left( \mathbb{Z},+\right) $ on $M$.

\item In general, if a group $G$ acts on $M$, then $\varphi \left( a,\cdot
\right) $ is a bijection for every $a\in G.$

\item If $G=\mathbb{Z}_{n}$, then $\varphi \left( 1,\cdot \right) $ is a
periodic mapping with a period $n$.
\end{enumerate}

In this paper we study uniformly Lipschitzian group actions.

\begin{definition}
A group (or a semigroup) $G$ is said to act uniformly $L$-Lipschitzly on a
metric space $M$ (\thinspace $L>0$) if
\begin{equation*}
\forall a\in G~\forall x,y\in M~d\left( \varphi \left( a,x\right) ,\varphi
\left( a,y\right) \right) \leq Ld\left( x,y\right) .
\end{equation*}
\end{definition}

\begin{remark}
If $\left( \mathbb{N},+\right) $ acts on $M$, then we obtain the usual
definition of a uniformly Lipschitzian mapping. If $G=\mathbb{Z}$ we recover
the definition of a uniformly bi-Lipschitzian mapping.
\end{remark}

From now on we shall write $\{T_{a}:a\in G\}$ for a representation of $G$ as
uniformly Lipschitzian mappings, i.e., $T_{a}x=\varphi \left( a,x\right)
,a\in G,x\in M.$ Define a metric
\begin{equation*}
d_{G}\left( x,y\right) =\sup_{a\in G}d\left( T_{a}x,T_{a}y\right) .
\end{equation*}%
Notice that the mappings $T_{a}$ are nonexpansive in this metric and even
isometries if $G$ is a group. The metric $d_{G}$ is equivalent to $d$ since
\begin{equation*}
d\left( x,y\right) \leq d_{G}\left( x,y\right) \leq Ld\left( x,y\right) .
\end{equation*}%
It follows that a group (a semigroup, resp.) acts uniformly Lipschitzly iff
it acts isometrically (nonexpansively, resp.) in an equivalent metric.

Recall that the orbit of the point $x\in M$ is the set
\begin{equation*}
O\left( x\right) =\left\{ T_{a}x:a\in G\right\}
\end{equation*}%
and its diameter $\delta (x)=\sup_{a,b\in G}d\left( T_{a}x,T_{b}x\right) .$
We say that the orbit is bounded if $\delta (x)<\infty .$ The radius of the
orbit $O\left( x\right) $ relative to a point $y\in M$ is defined by
\begin{equation*}
r\left( y,x\right) =\sup_{a\in G}d\left( y,T_{a}x\right)
\end{equation*}%
and the radius of $O\left( x\right) $ is given by
\begin{equation*}
r\left( x\right) =\inf_{y\in M}r\left( y,x\right) .
\end{equation*}%
It is easy to see that%
\begin{equation*}
\frac{\delta (x)}{2}\leq r\left( x\right) \leq \delta (x).
\end{equation*}%
Notice that if a group $G$ acts on $M$, then the orbits are disjoint or
equal, in other words they form the equivalence classes of this action.

\begin{lemma}
\label{bounded}If $\{T_{a}:a\in G\}$ is a group (or a semigroup) of
uniformly Lipschitzian mappings, then all orbits are simultaneously bounded
or unbounded.
\end{lemma}

\begin{proof}
Assume that $\delta (x)<\infty $ for some $x\in M.$ Then, for every $y\in M$,%
\begin{align*}
\delta \left( y\right) & =\sup_{a,b\in G}d\left( T_{a}y,T_{b}y\right) \leq
\sup_{a,b\in G}(d\left( T_{a}y,T_{a}x\right) +d\left( T_{a}x,T_{b}x\right)
+d\left( T_{b}x,T_{b}y\right) ) \\
& \leq 2Ld\left( x,y\right) +\delta \left( x\right) <\infty .
\end{align*}
\end{proof}

The following definition is central for our work.

\begin{definition}
The center of the orbit of $x\in M$ is the set
\begin{equation*}
C\left( x\right) =\left\{ y\in M:r\left( y,x\right) =r\left( x\right)
\right\}
\end{equation*}%
and the center of $C(x)$ is defined by
\begin{equation*}
CC\left( x\right) =\bigcap_{y\in C\left( x\right) }B\left( y,r\left(
x\right) \right) \cap C\left( x\right) .
\end{equation*}
\end{definition}

\section{Common fixed points}

\begin{definition}
A point $x_{0}\in M$ is a common fixed point for a group $\{T_{a}:a\in G\}$
if
\begin{equation*}
\forall a\in G\text{ }T_{a}x_{0}=x_{0}.
\end{equation*}%
The set of common fixed points is denoted by $\Fix G.$
\end{definition}

Notice that if $\Fix G$ is nonempty and the group (semigroup) acts uniformly
Lipschitzly, then all orbits are bounded, see Lemma \ref{bounded}, i.e., the
boundedness of orbits is a necessary condition for the existence of common
fixed points in this case.

In this section we prove theorems concerning the existence of common fixed
points for uniformly Lipschitzian group actions on hyperconvex spaces. First
recall some basic facts.

\begin{definition}
A metric space $(M,d)$ is called hyperconvex if
\begin{equation*}
\bigcap_{\alpha \in \Gamma }B(x_{\alpha },r_{\alpha })\neq \emptyset
\end{equation*}%
for any collection of closed balls $\{B(x_{\alpha },r_{\alpha })\}_{\alpha
\in \Gamma }$ such that $d(x_{\alpha },x_{\beta })\leq r_{\alpha }+r_{\beta
},\ \alpha ,\beta \in \Gamma .$
\end{definition}

It is not difficult to see that hyperconvex spaces are complete. We will use
this fact several times. In hyperconvex spaces,%
\begin{equation*}
r\left( x\right) =\frac{\delta \left( x\right) }{2},
\end{equation*}%
whenever the orbits are bounded. The sets $C\left( x\right) ,CC\left(
x\right) $ are nonempty since $\diam C\left( x\right) \leq 2r\left( x\right)
.$ Furthermore,
\begin{equation*}
C\left( x\right) =\bigcap_{a\in G}B\left( T_{a}x,r\left( x\right) \right) .
\end{equation*}

\begin{theorem}
\label{1}If $\{T_{a}:a\in G\}$ is a group of uniformly $L$-Lipschitzian
mappings on a hyperconvex space $M$ with $L<\sqrt{2}$ and the orbits $%
O\left( x\right) $ are bounded, then $\Fix G$ is nonempty.
\end{theorem}

\begin{proof}
Without loss of generality we can assume that $L\in \left( 1,\sqrt{2}\right)
.$ Fix $x_{1}\in M$ and select $x_{2}\in CC\left( x_{1}\right) .$ Then%
\begin{equation*}
\forall a,b\in G\ d\left( T_{a}x_{2},T_{b}x_{1}\right) \leq Ld\left(
x_{2},T_{a^{-1}b}x_{1}\right) \leq Lr\left( x_{1}\right) .
\end{equation*}%
Now notice that from hyperconvexity, for every $a\in G$ there exists $%
y_{a}\in M$ such that
\begin{equation*}
y_{a}\in \bigcap_{b\epsilon G}B\left( T_{b}x_{1},r\left( x_{1}\right)
\right) \cap B\left( T_{a}x_{2},\left( L-1\right) r\left( x_{1}\right)
\right) .
\end{equation*}%
Hence $y_{a}\in C\left( x_{1}\right) $ and
\begin{equation*}
d\left( y_{a},T_{a}x_{2}\right) \leq \left( L-1\right) r\left( x_{1}\right) .
\end{equation*}%
Then%
\begin{equation*}
d\left( x_{_{2}},T_{a}x_{2}\right) \leq d\left( x_{2},y_{a}\right) +d\left(
y_{a},T_{a}x_{2}\right) \leq r\left( x_{1}\right) +\left( L-1\right) r\left(
x_{1}\right) =Lr\left( x_{1}\right)
\end{equation*}%
and it follows that
\begin{equation*}
d\left( T_{a}x_{2},T_{b}x_{2}\right) \leq Ld\left(
x_{2},T_{a^{-1}b}x_{2}\right) \leq L^{2}r\left( x_{1}\right)
\end{equation*}%
for every $a,b\in G.$ Hence
\begin{equation*}
\delta \left( x_{2}\right) \leq \frac{L^{2}}{2}\delta \left( x_{1}\right) .
\end{equation*}%
Next we select $x_{3}\in CC\left( x_{2}\right) $ and estimate $\delta \left(
x_{3}\right) $ in a similar way. We continue in this fashion obtaining
recursively a sequence $(x_{n})$ such that $\delta \left( x_{n+1}\right)
\leq \frac{L^{2}}{2}\delta \left( x_{n}\right) $ and $\delta \left(
x_{n+1},x_{n}\right) =\frac{\delta \left( x_{n}\right) }{2}.$ It follows
that $(x_{n})$ is a Cauchy sequence converging to a point $x_{0}\in \Fix G$
since%
\begin{equation*}
\delta \left( x_{0}\right) \leq 2Ld\left( x_{0},x_{n}\right) +\delta \left(
x_{n}\right) \rightarrow 0.
\end{equation*}
\end{proof}

In the above theorem we do not assume the boundedness of $M$ but the
boundedness of orbits, only. The following example of S. Prus (see \cite[p.
412]{KiSi}) shows that in the case of semigroups the boundedness of orbits
does not imply the existence of a fixed point even for nonexpansive mappings.

\begin{example}
Let $M=\ell _{\infty }$ and consider the semigroup generated by the mapping $%
T:\ell _{\infty }\rightarrow \ell _{\infty }$ defined by%
\begin{equation*}
T(x_{n})=\left( 1+\mathrm{LIM}~x_{n},x_{1},x_{2},...\right) ,
\end{equation*}%
where $\mathrm{LIM}$ denotes a Banach limit. Then $\left\Vert
T^{n}x\right\Vert \leq 1+\left\Vert x\right\Vert $ and
\begin{equation*}
\left\Vert T^{n}x-T^{n}y\right\Vert =\left\Vert x-y\right\Vert
\end{equation*}%
for every $x,y\in \ell _{\infty }.$ If $\bar{x}\in \ell _{\infty }$
satisfies $T\bar{x}=\bar{x}$, then%
\begin{equation*}
\bar{x}_{1}=1+\mathrm{LIM}~\bar{x}_{n},\ \bar{x}_{2}=\bar{x}_{1},\,\bar{x}%
_{3}=\bar{x}_{2},...
\end{equation*}%
and hence $\mathrm{LIM}~\bar{x}_{n}=1+\mathrm{LIM}~\bar{x}_{n},$ a
contradiction, which shows that $\Fix T=\emptyset .$
\end{example}

An interesting special case of Theorem \ref{1}, proved independently in \cite%
{La} (see also \cite{Dr2}), is concerned with the group of all surjective
isometries on a bounded hyperconvex space $M.$

\begin{corollary}
If $\{T_{a}:a\in G\}$ is a group of isometries with bounded orbits $O\left(
x\right) $ on a hyperconvex space $M$, then $\Fix G$ is nonempty. In
particular, if $M$ is bounded, there exists a common fixed point for all
surjective isometries on $M$.
\end{corollary}

If $M$ is unbounded, it is not difficult to find two surjective isometries
without a common fixed point, however the group which they generate has
unbounded orbits.

\begin{example}
Consider the space $\mathbb{R}^{2}$ with the maximum norm and let $%
T_{a},T_{b}$ be rotations around two distinct points $a,b\in \mathbb{R}^{2}$
through the angle $\theta =\frac{\pi }{2}.$ Each of them is clearly a
surjective isometry with bounded orbits but the composition $T_{a}\circ T_{b}
$ (and hence the whole group) has unbounded orbits.
\end{example}

Now we give an analogous theorem for $\lambda $-hyperconvex spaces.

\begin{definition}
A subset $D$ of a metric space $M$ is called admissible if it is the
intersection of closed balls.
\end{definition}

\begin{definition}
A metric space $\left( M,d\right) $ is said to be $\lambda $-hyperconvex if
for every non-empty admissible set $D$ and for any family of closed balls $%
\{B(x_{\alpha },r_{\alpha })\}_{\alpha \in \Gamma }$ centered at $x_{\alpha
}\in D,\alpha \in \Gamma ,$ such that $d(x_{\alpha },x_{\beta })\leq
r_{\alpha }+r_{\beta },\ \alpha ,\beta \in \Gamma ,$ the intersection
\begin{equation*}
D\cap \bigcap_{\alpha \in \Gamma }B\left( x_{\alpha },\lambda r_{\alpha
}\right) \neq \emptyset .
\end{equation*}
\end{definition}

Like hyperconvex, $\lambda $-hyperconvex spaces are complete. It is known
that $M$ is hyperconvex iff it is $1$-hyperconvex. If $\{T_{a}:a\in G\}$ is
a group of mappings acting on a $\lambda $-hyperconvex space, we have the
following estimations:
\begin{equation*}
\frac{\delta \left( x\right) }{2}\leq r\left( x\right) \leq \lambda \frac{%
\delta \left( x\right) }{2}.
\end{equation*}%
However, the center $C\left( x\right) $ may be empty and instead we will use
the sets%
\begin{align*}
A\left( x\right) & =\bigcap_{a\in G}B\left( T_{a}x,\frac{\lambda \delta
\left( x\right) }{2}\right) , \\
AA\left( x\right) & =A\left( x\right) \cap \bigcap_{y\in A\left( x\right)
}B\left( y,\frac{\lambda ^{2}\delta \left( x\right) }{2}\right)
\end{align*}%
which are nonempty by the definition of $\lambda $-hyperconvexity.

\begin{theorem}
\label{2}If $\{T_{a}:a\in G\}$ is a group of uniformly $L$-Lipschitzian
mappings on a $\lambda $-hyperconvex space $M$ with $L<\frac{\sqrt{2}}{%
\lambda }$ and the orbits $O\left( x\right) $ are bounded, then $\Fix G$ is
nonempty.
\end{theorem}

\begin{proof}
We can assume that $L\in \left( 1,\frac{\sqrt{2}}{\lambda }\right) $. Fix $%
x_{1}\in M$ and select $x_{2}\in AA\left( x_{1}\right) .$ Then%
\begin{equation*}
\forall a,b\in G~d\left( T_{a}x_{2},T_{b}x_{1}\right) \leq Ld\left(
x_{2},T_{a^{-1}b}x_{1}\right) \leq \frac{L\lambda \delta \left( x_{1}\right)
}{2}.
\end{equation*}%
It follows from $\lambda $-hyperconvexity that for every $a\in G$ there
exists $y_{a}\in M$ such that
\begin{equation*}
y_{a}\in \bigcap_{b\in G}B\left( T_{b}x_{1},\frac{\lambda \delta \left(
x_{1}\right) }{2}\right) \cap B\left( T_{a}x_{2},\frac{\left( L-1\right)
\lambda ^{2}\delta \left( x_{1}\right) }{2}\right) .
\end{equation*}%
Hence $y_{a}\in A\left( x_{1}\right) $ and
\begin{equation*}
d\left( y_{a},T_{a}x_{2}\right) \leq \frac{\left( L-1\right) \lambda
^{2}\delta \left( x_{1}\right) }{2}.
\end{equation*}%
Then
\begin{equation*}
d\left( x_{_{2}},T_{a}x_{2}\right) \leq d\left( x_{2},y_{a}\right) +d\left(
y_{a},T_{a}x_{2}\right) \leq \frac{\lambda ^{2}\delta \left( x_{1}\right) }{2%
}+\frac{\left( L-1\right) \lambda ^{2}\delta \left( x_{1}\right) }{2}=\frac{%
L\lambda ^{2}\delta \left( x_{1}\right) }{2}
\end{equation*}%
and hence%
\begin{equation*}
d\left( \varphi \left( a,x_{2}\right) ,\varphi \left( b,x_{2}\right) \right)
\leq Ld\left( x_{2},\varphi \left( b-a,x_{2}\right) \right) \leq \frac{%
L^{2}\lambda ^{2}\delta \left( x_{1}\right) }{2}
\end{equation*}%
for every $a,b\in G$, which gives%
\begin{equation*}
\delta \left( x_{2}\right) \leq \frac{L^{2}\lambda ^{2}}{2}\delta \left(
x_{1}\right) .
\end{equation*}%
Now we select $x_{3}\in AA\left( x_{2}\right) $ and estimate $\delta \left(
x_{3}\right) $ analogously. Thus we obtain recursively a sequence $(x_{n})$
such that
\begin{equation*}
\delta \left( x_{n+1}\right) \leq \frac{L^{2}\lambda ^{2}}{2}\delta \left(
x_{n}\right) ,
\end{equation*}%
\begin{equation*}
d\left( x_{n+1},x_{n}\right) \leq \frac{\lambda \delta \left( x_{n}\right) }{%
2}\leq \frac{\lambda }{2}\left( \frac{L^{2}\lambda ^{2}}{2}\right)
^{n-1}\delta \left( x_{1}\right) .
\end{equation*}%
It follows that $(x_{n})$ is a Cauchy sequence (since $\frac{L^{2}\lambda
^{2}}{2}<1$) which converges to a point $x_{0}\in \Fix G.$
\end{proof}

\begin{corollary}
If $\{T_{a}:a\in G\}$ is a group of isometries on a $\lambda $-hyperconvex
space $M$ with $\lambda <\sqrt{2}$ and the orbits are bounded, then $\Fix G$
is nonempty. In particular, if $M$ is bounded, there exists a common fixed
point for all surjective isometries on $M$.
\end{corollary}

A natural example of group actions (where $G=\mathbb{Z}_{n}$) are $n$%
-periodic mappings. Uniformly Lipschitzian $n$-periodic mappings were
studied by W. A. Kirk in \cite{Ki}, where the results for a constant $L$
satisfying $n^{-2}[(n-1)(n-2)L^{2}+2(n-1)L]<1$ were obtained in any bounded,
closed, convex subset of a Banach space (see also \cite{GoPu}). Here we
obtain the results independent of $n$ and for some metric spaces that are
not necessarily bounded.

If a group $\mathbb{Z}_{2}$ acts on a metric space, the orbits consist of
two elements (so they are always bounded). In this case $T^{2}=I$ and such a
mapping $T$ is called an involution. In the following theorem we consider $%
\lambda $-hyperconvex spaces.

\begin{theorem}
\label{3}If $T$ is an $L$-Lipschitzian involution in a $\lambda $%
-hyperconvex space with $L<\frac{2}{\lambda ^{2}},$ then $\Fix T\neq
\emptyset $.
\end{theorem}

\begin{proof}
In the case of an involution,%
\begin{align*}
\delta \left( x\right) & =d\left( x,Tx\right) , \\
A\left( x\right) & =B\left( x,\frac{\lambda \delta \left( x\right) }{2}%
\right) \cap B\left( Tx,\frac{\lambda \delta \left( x\right) }{2}\right) , \\
AA\left( x\right) & =\bigcap_{y\in A\left( x\right) }B\left( y,\frac{\lambda
^{2}\delta \left( x\right) }{2}\right) \cap A\left( x\right) .
\end{align*}%
Fix $x_{1}\in M$ and select $x_{2}\in AA\left( x\right) .$ Then
\begin{equation*}
d\left( Tx_{2},x_{1}\right) \leq \frac{L\lambda \delta \left( x_{1}\right) }{%
2},
\end{equation*}%
and%
\begin{equation*}
d\left( Tx_{2},Tx_{1}\right) \leq \frac{L\lambda \delta \left( x_{1}\right)
}{2}.
\end{equation*}%
It follows from the definition of $\lambda $-hyperconvexity that there
exists $y$ such that
\begin{equation*}
y\in B\left( x_{1},\frac{\lambda \delta \left( x_{1}\right) }{2}\right) \cap
B\left( \varphi (x_{1}),\frac{\lambda \delta \left( x_{1}\right) }{2}\right)
\cap B\left( \varphi \left( x_{2}\right) ,\frac{\left( L-1\right) \lambda
^{2}\delta \left( x_{1}\right) }{2}\right) .
\end{equation*}%
Hence
\begin{equation*}
d\left( x_{2},Tx_{2}\right) \leq d\left( x_{2},y\right) +d\left(
y,Tx_{2}\right) \leq \frac{L\lambda ^{2}\delta \left( x_{1}\right) }{2}.
\end{equation*}%
As in the proof of Theorem \ref{2} we have a sequence $(x_{n})$ such that%
\begin{align*}
\delta \left( x_{n+1}\right) & \leq \frac{L\lambda ^{2}}{2}\delta \left(
x_{n}\right) , \\
d\left( x_{n+1},x_{n}\right) & \leq \frac{\lambda \delta \left( x_{n}\right)
}{2}\leq \frac{\lambda }{2}\left( \frac{L\lambda ^{2}}{2}\right)
^{n-1}\delta \left( x_{1}\right) .
\end{align*}%
It is enough to put $L<\frac{2}{\lambda ^{2}}$ to guarantee that $(x_{n})$
is a Cauchy sequence.
\end{proof}

\begin{corollary}
Every $L$-Lipschitzian involution in a hyperconvex space with $L<2$ has a
fixed point.
\end{corollary}

The above corollary should be compared with the result of K. Goebel and E. Z%
\l otkiewicz \cite{GoZl} who proved that if C is a closed and convex subset
of a Banach space, then every $L$-Lipschitzian involution $T:C\rightarrow C$
with $L<2$ has a fixed point. Our result is the analogue for hyperconvex
metric spaces.

\section{Retractions onto $\Fix G$}

It was proved in \cite{RaSi} (see also \cite{Ba}) that the set of fixed
points of a nonexpansive mapping in hyperconvex spaces is itself hyperconvex
and hence is a nonexpansive retract of the domain. This is no longer true if
a mapping is Lipschitzian with a constant $k>1.$

\begin{example}
Let $a\in (0,1)$ and put
\begin{equation*}
f_{a}(t)=\left\{
\begin{array}{c}
\begin{array}{cc}
\left( 1+a\right) t+a, & t\in \left[ -1,0\right] \\
\left( 1-a\right) t+a, & t\in \left( 0,1\right] .%
\end{array}%
\end{array}%
\right.
\end{equation*}%
Define a mapping $T:B_{\ell _{\infty }}\rightarrow B_{\ell _{\infty }}$ by%
\begin{equation*}
Tx=\left( f_{a}(x_{1}),f_{a}(x_{2}),...\right) .
\end{equation*}%
Then
\begin{equation*}
\Fix T=\left\{ x\in B:x_{n}\in \{0,1\}\text{ for every }n\in \mathbb{N}%
\right\} .
\end{equation*}%
The Lipschitz constant of $T$ equals $1+a$ which is close to $1$ for $a$
close to $0$. Furthermore, $T$ is a bi-Lipschitz bijection and hence we can
define a $\mathbb{Z}$-action on $B_{\ell _{\infty }}$ which is Lipschitzian
(but not uniformly Lipschitzian). Notice that $\Fix T$ is not a (continuous)
retract of $B_{\ell _{\infty }}$ since it consists of points of distance $2$
apart.
\end{example}

We show that if a group action on a hyperconvex space $M$ is $L$-uniformly
Lipschitzian with $L<\sqrt{2}$, then the set of common fixed points is a H%
\"{o}lder continuous retract of $M.$ We begin with a few lemmas concerned
with the radius and the center of a subset $K$ of a metric space $M$. Let%
\begin{align*}
r\left( y,K\right) & =\sup_{x\in A}d\left( y,x\right) , \\
r\left( K\right) & =\inf_{y\in M}\sup_{x\in K}d\left( y,x\right) , \\
C\left( K\right) & =\left\{ y\in M:\sup_{x\in K}d\left( y,x\right) =r\left(
K\right) \right\}
\end{align*}%
denote the radius of $K$ relative to $y,$ the radius of $K$ and the center
of $K$ relative to $M$, respectively. In general, $C\left( K\right) $ may be
empty. The Hausdorff distance of bounded subsets $K$ and $L$ is defined by
\begin{equation*}
D\left( K,L\right) =\max \{\sup_{x\in K}\inf_{y\in L}d\left( y,x\right)
,\sup_{x\in L}\inf_{y\in K}d\left( y,x\right) \}.
\end{equation*}%
If
\begin{equation*}
K_{r}=\bigcup_{x\in A}B\left( x,r\right) ,
\end{equation*}%
then
\begin{equation*}
D\left( K,L\right) =\inf \left\{ r>0:K\subset L_{r}\wedge L\subset
K_{r}\right\} .
\end{equation*}

\begin{lemma}
\label{radius}For bounded subsets $K,L$ of a metric space $M,$%
\begin{equation*}
\left\vert r\left( K\right) -r\left( L\right) \right\vert \leq D\left(
K,L\right) .
\end{equation*}
\end{lemma}

\begin{proof}
Fix $\varepsilon >0.$ There exist $p,q\in M$ such that%
\begin{equation*}
K\subset B\left( p,r\left( K\right) +\varepsilon \right) ,\ L\subset B\left(
q,r\left( L\right) +\varepsilon \right)
\end{equation*}%
since
\begin{equation*}
\forall x\in K\ \exists y\in L\ d\left( x,y\right) \leq D\left( K,L\right)
+\varepsilon .
\end{equation*}%
Hence%
\begin{equation*}
r\left( q,K\right) \leq r\left( L\right) +\varepsilon +D\left( K,L\right)
+\varepsilon
\end{equation*}%
which gives
\begin{equation*}
r\left( K\right) \leq r\left( L\right) +D\left( K,L\right) +2\varepsilon .
\end{equation*}%
In a similar way we show that
\begin{equation*}
r\left( L\right) \leq r\left( K\right) +D\left( K,L\right) +2\varepsilon .
\end{equation*}%
This completes the proof since $\varepsilon $ is arbitrary.
\end{proof}

\begin{lemma}
\label{center}In hyperconvex spaces,
\begin{equation*}
D\left( C\left( K\right) ,C\left( L\right) \right) \leq D\left( K,L\right)
+\left\vert r\left( K\right) -r\left( L\right) \right\vert .
\end{equation*}
\end{lemma}

\begin{proof}
Notice first that $C\left( K\right) ,C\left( L\right) $ are nonempty in a
hyperconvex space. Fix $\varepsilon >0$ and $q\in C\left( L\right) .$ From
the definition of the Hausdorff metric,
\begin{equation*}
\forall x\in K\ \exists y\in L\ d\left( x,y\right) \leq D\left( K,L\right)
+\varepsilon
\end{equation*}%
which gives
\begin{equation*}
\forall x\in K\ d\left( x,q\right) \leq D\left( K,L\right) +r\left( L\right)
+\varepsilon .
\end{equation*}%
It follows that there exists $p\in C\left( K\right) $ such that
\begin{equation*}
p\in \bigcap_{x\in K}B\left( x,r\left( K\right) \right) \cap B\left(
q,D\left( K,L\right) +r\left( L\right) -r\left( K\right) +\varepsilon
\right) .
\end{equation*}%
Hence
\begin{equation*}
\forall q\in C\left( L\right) \ \exists p\in C\left( K\right) \ d\left(
p,q\right) \leq D\left( K,L\right) +r\left( L\right) -r\left( K\right)
+\varepsilon .
\end{equation*}%
Analogously,
\begin{equation*}
\forall p\in C\left( K\right) \ \exists q\in C\left( L\right) \ d\left(
p,q\right) \leq D\left( K,L\right) +r\left( K\right) -r\left( L\right)
+\varepsilon .
\end{equation*}%
This completes the proof since $\varepsilon $ is arbitrary.
\end{proof}

\begin{corollary}
\label{est}In hyperconvex spaces,%
\begin{equation*}
D\left( C\left( K\right) ,C\left( L\right) \right) \leq 2D\left( K,L\right) ,
\end{equation*}%
\begin{equation*}
r\left( K\right) =r\left( L\right) \Rightarrow D\left( C\left( K\right)
,C\left( L\right) \right) =D\left( K,L\right) .
\end{equation*}
\end{corollary}

\begin{example}
Consider the space $\mathbb{R}^{2}$ with the maximum norm and let $K=\left\{
\left( 0,0\right) \right\} ,L=\left\{ \left( t,1\right) ,t\in \left[ -1,1%
\right] \right\} .$ We have $C\left( K\right) =\left\{ \left( 0,0\right)
\right\} ,C\left( L\right) =\left\{ \left( 0,t\right) ,t\in \left[ 0,2\right]
\right\} $. Then
\begin{equation*}
D\left( C\left( K\right) ,C\left( L\right) \right) =2=2D\left( K,L\right)
\end{equation*}%
which shows that the estimation in Corollary \ref{est} is sharp.
\end{example}

\begin{remark}
In general spaces we have no similar estimation. Even in a Hilbert space,
the distance of centers $D\left( C\left( K\right) ,C\left( L\right) \right) $
does not depend on the distance $D\left( K,L\right) $ in a Lipschitz way,
see \cite[Theorem 1]{SzVl}, \cite[Theorem 7]{AlPa}. The following example
shows that there is no hope to extend Lemma \ref{center} to $\lambda $%
-hyperconvex spaces.
\end{remark}

\begin{example}
\label{2-hyp}Let $M=\left\{ z=e^{i\varphi },\varphi \in \left[ 0,2\pi
\right) \right\} $ with the metric
\begin{equation*}
d\left( z_{1},z_{2}\right) =d\left( e^{i\varphi },e^{i\psi }\right) =\left\{
\begin{array}{c}
\left\vert \varphi -\psi \right\vert ,\ \text{if }\left\vert \varphi -\psi
\right\vert \leq \pi , \\
2\pi -\left\vert \varphi -\psi \right\vert ,\text{\ if }\left\vert \varphi
-\psi \right\vert >\pi .%
\end{array}%
\right.
\end{equation*}%
Note that this is a $2$-hyperconvex space. Fix $\varepsilon >0$ and let $%
K=\left\{ 1,-1\right\} =\left\{ e^{0},e^{i\pi }\right\} ,L=\left\{
e^{i\left( \pi -\varepsilon \right) },1\right\} .$ Then $C\left( K\right)
=\left\{ i,-i\right\} =\left\{ e^{i\frac{\pi }{2}},e^{-i\frac{\pi }{2}%
}\right\} ,C\left( L\right) =\left\{ e^{i\frac{\pi -\varepsilon }{2}%
}\right\} $ and hence%
\begin{align*}
D\left( C\left( K\right) ,C\left( L\right) \right) & =\pi -\frac{\varepsilon
}{2}, \\
D\left( K,L\right) & =\varepsilon .
\end{align*}%
Taking $\varepsilon $ close to $0,$ it follows that there is no continuous
dependence of $D\left( C\left( K\right) ,C\left( L\right) \right) $ on $%
D\left( K,L\right) $ in this case.
\end{example}

We conclude with proving a qualitative version of Theorem \ref{1}. The proof
relies on the following two results. The first one is a special case of \cite%
[Theorem 1]{KKM}. Let $\mathcal{A}(M)$ denote the family of all nonempty
admissible subsets of $M$.

\begin{theorem}
\label{selection}Let $M$ be a hyperconvex metric space, let $S$ be any set,
and let $T^{\ast }:S\rightarrow \mathcal{A}(M).$ Then there exists a mapping
$T:S\rightarrow M$ for which $Tx\in T^{\ast }x$ for $x\in S$ and for which $%
d(Tx,Ty)\leq D(T^{\ast }x,T^{\ast }y)$ for each $x,y\in S$.
\end{theorem}

\begin{theorem}[{see, e.g., {\protect\cite[Prop. 1.10]{BeLi}, \protect\cite[%
Lemma 2.2]{Wi}}}]
\label{holder} Let $(X,d)$ be a complete bounded metric space and let $%
f:X\rightarrow X$ be a $k$-Lipschitzian mapping. Suppose there exists $%
0<\gamma <1$ and $c>0$ such that $d(f^{n+1}(x),f^{n}(x))\leq c\gamma ^{n}$
for every $x\in X$. Then $Rx=\lim_{n\rightarrow \infty }f^{n}(x)$ is a H\"{o}%
lder continuous mapping.
\end{theorem}

If we analyze the proof of Theorem \ref{1}, we see that the sets $C\left(
x\right) ,CC\left( x\right) $ are admissible. Furthermore, if a group acts
uniformly $L$-Lipschitzly on a hyperconvex space, then by Corollary \ref{est}%
,%
\begin{align*}
D\left( O\left( x\right) ,O\left( y\right) \right) & \leq Ld\left(
x,y\right) , \\
D\left( C\left( x\right) ,C\left( y\right) \right) & \leq 2Ld\left(
x,y\right) , \\
D\left( CC\left( x\right) ,CC\left( y\right) \right) & \leq 4Ld\left(
x,y\right)
\end{align*}%
and we can use Theorem \ref{selection} to obtain a $4L$-Lipschitzian
selection $f:M\rightarrow M$ such that $f(x)\in CC\left( x\right) ,x\in M.$
This leads to the following result.

\begin{theorem}
If $\{T_{a}:a\in G\}$ is a group of uniformly $L$-Lipschitzian mappings on a
hyperconvex space $M$ with $L<\sqrt{2}$ and the orbits are bounded, then $%
\Fix G$ is a H\"{o}lder continuous retract of $M.$
\end{theorem}

\begin{proof}
Let $L\in \left( 1,\sqrt{2}\right) .$ The observation given above gives a $%
4L $-Lipschitzian mapping $f:M\rightarrow M$ such that $f\left( x\right) \in
CC\left( x\right) $ for each $x\in M.$ Fix $\bar{x}\in M$ and let $x_{1}=f(%
\bar{x})\in CC\left( \bar{x}\right) .$ Now we can follow the proof of
Theorem \ref{1} to get
\begin{equation*}
\delta \left( f(\bar{x})\right) \leq \frac{L^{2}}{2}\delta \left( \bar{x}%
\right) .
\end{equation*}%
Applying this argument again, we obtain recursively a sequence $(f^{n}(\bar{x%
}))$ such that $\delta \left( f^{n+1}(\bar{x})\right) \leq \frac{L^{2}}{2}%
\delta \left( f^{n}(\bar{x})\right) $ and $d\left( f^{n+1}(\bar{x}),f^{n}(%
\bar{x})\right) =\frac{\delta \left( f^{n}(\bar{x})\right) }{2}$. It follows
that $(f^{n}(\bar{x}))$ is a Cauchy sequence and we can define a mapping%
\begin{equation*}
R\bar{x}=\lim_{n\rightarrow \infty }f^{n}(\bar{x})
\end{equation*}%
for every $\bar{x}\in M.$ It is not difficult to see that $T_{a}R\bar{x}=R%
\bar{x}$ for every $a\in G$ and $\bar{x}\in M.$ Furthermore, $R\bar{x}=\bar{x%
}$ if $\bar{x}\in \Fix G$ since then $f(\bar{x})=\bar{x}.$ Using Theorem \ref%
{holder} it follows that $R:M\rightarrow \Fix G$ is a H\"{o}lder continuous
retraction onto $\Fix G.$
\end{proof}

\begin{acknowledgement}
The authors are grateful to Rafael Esp\'{\i}nola and the referee for helpful
comments on the manuscript and drawing their attention to the paper of U.
Lang \cite{La}.
\end{acknowledgement}

\end{document}